\documentclass{amsart}
\usepackage[T1]{fontenc}
\usepackage{amssymb, amsmath, amsthm}

\usepackage[backend=biber]{biblatex} %
\newtheorem{theorem}[subsection]{Theorem}

\newtheorem{corollary}[subsection]{Corollary}

\newtheorem{example}[subsection]{Example}

\newtheorem{proposition}[subsection]{Proposition}

\newtheorem{definition}[subsection]{Definition}

\newtheorem{lemma}[subsection]{Lemma}

\newtheorem{question}[subsection]{Question}

\newtheorem{problem}[subsection]{Problem}

\title{A note on the Macías topology}

\author{Jhixon Macías and Reyes Ortiz}

\address{Jhixon Macías\newline
University of Puerto Rico at Mayaguez, Mayaguez, PR, USA\newline
United States of America}
\email{jhixon.macias@upr.edu}

\address{Reyes Ortiz\newline
University of Puerto Rico at Mayaguez, Mayaguez, PR, USA\newline
United States of America}
\email{reyes.ortiz@upr.edu}


\begin{document}

\begin{abstract}
We study some properties of the closure operator in the Macías topology on infinite integral domains. Moreover, under certain conditions, we present topological proofs of the infiniteness of maximal ideals and non-associated irreducible elements, taking advantage of the hyperconnectedness of the Macías topology. Additionally, some problems are proposed.
\end{abstract}

\maketitle
\section{Introduction}\label{section1}
In 1955, H. Furstenberg introduced a topology $\tau_F$ on the integers $\mathbb{Z}$ generated by the collection of arithmetic progressions of the form $a+b\mathbb{Z}$ ($a,b\in \mathbb{Z}, b\geq 1$) and with it presented the first topological proof of the infinitude of prime numbers, see \cite{furstenberg1955infinitude}. A couple of years earlier (1953), M. Brown had introduced a topology $\tau_G$ on the natural numbers $\mathbb{N}$ that turns out to be coarser than Furstenberg's topology induced on $\mathbb{N}$. This latter topology is generated by the collection of arithmetic progressions $a+b\mathbb{N}$ with $a\in \mathbb{N}$ and $b\in\mathbb{N}\cup\{0\}$ such that $\gcd(a,b)=1$. It was not until 1959 that the topology introduced by Brown was popularized by S. Golomb (now known as Golomb's topology) who in \cite{golomb1959connected} proves that Dirichlet's theorem on arithmetic progressions is equivalent to the set of prime numbers $\mathbb{P}$ being dense in the topological space $(\mathbb{N},\tau_G)$. Golomb also presents a topological proof of the infinitude of prime numbers, verifying that Furstenberg's argument also works on $(\mathbb{N},\tau_G)$. In 1997, a \textit{Golomb topology} on any integral domain was defined by J. Knopemacher and S. Porubsky \cite{knopemacher1997topologies} and is also considered by  P. L. Clark in \cite{clark2017euclidean}.  In both cases, the proofs of Golomb and Furstenberg are generalized, respectively. On the one hand, J. Knopfmacher and S. Porubsky demonstrate that any integral domain (that is not a field) has an infinite number of maximal ideals if the set of its units is not open in its respective \textit{Golomb space}, and on the other hand, P. L. Clark proves that a semiprimitive integral domain (that is not a field) in which every nonzero nonunit element has at least one irreducible divisor has an infinite number of non-associated irreducible elements.

In \cite{Jhixon2023} the \textit{Macías topology} $\tau_M\subset\tau_{G}$ is introduced, which is generated by the collection of sets $\sigma_n:=\{m\in\mathbb{N}: \gcd(n,m)=1\}$. Additionally, some of its properties are studied. For example, the topological space $(\mathbb{N},\tau_M)$ does not satisfy the $\mathrm{T}_0$ separation axiom, satisfies (vacuously) the $\mathrm{T}_4$ separation axiom, is ultraconnected, hyperconnected (therefore connected, locally connected, path-connected), not countably compact (therefore not compact), but it is limit point compact and pseudocompact. On the other hand, for $n,m\in\mathbb{N}$, it holds that $\textbf{cl}_{(\mathbb{N},\tau_M)}(\{nm\})=\textbf{cl}_{(\mathbb{N},\tau_M)}(\{n\})\cap \textbf{cl}_{(\mathbb{N},\tau_M)}(\{m\})$ where $\textbf{cl}_{(\mathbb{N},\tau_M)}(\{n\})$ denotes the closure of the singleton set $\{n\}$ in the topological space $(\mathbb{N},\tau_M)$. Furthermore, in \cite{jhixon2024} (using the topology $\tau_M$), a topological proof of the infinitude of prime numbers is presented (a different from the proofs of Furstenberg and Golomb). In the same work, the infinitude of any non-empty subset of prime numbers is characterized in the following sense: if $A\subset \mathbb{P}$ (non-empty), then $A$ is infinite if and only if $A$ is dense in $(\mathbb{N},\tau_M)$, see \cite[Theorem 4]{jhixon2024}. The \textit{Macías topology} on any integral domain was defined by J. Macías \cite{MACIAS2024109070}, and some of its properties were studied (particularly on principal ideal domains). Moreover, a generalization of Macías' proof was presented in the following sense: any principal ideal domain (that is not a field) has an infinite number of maximal ideals if the set of its units does not form an open set in the \textit{Macías space}.

In this work, we study some properties of the closure operator in the Macías topology on infinite integral domains (Section \ref{section3}). In particular, we characterize when the closure of an element is the ideal generated by it (Theorem \ref{th1}), as well as topologically characterizing when a unique factorization domain is a principal ideal domain (Corollary \ref{corcara}). We also proved that $\widetilde{M(R)}$ and $M(R)$ are ultraconnected. Additionally, at the end of this section, we propose some open problems. Moreover, in Section \ref{section4}, under certain conditions, we present two topological proofs of the infinitude of maximal ideals on infinite integral domains (Subsection \ref{subsection4.1}) and non-associated irreducible elements on infinite \textit{Furstenberg domains} (Subsection \ref{subsection4.2}), taking advantage of the hyperconnectedness of the Macías topology.

\section{Preliminaries}\label{section2}

Throughout this paper, we denote by $R$ an integral domain that is not a field (hence infinite).  If an additional property on $R$ is required, it will be mentioned as appropriate. In general, for $X\subset R$, $X^0=X\setminus\{0\}$. On the other hand, let us denote by $R^\times$ and $R^\#$ the set of units of $R$ and the set $R^0 \setminus R^\times$, respectively. 

For each $r\in R$, define $\sigma_r:=\{s\in R: \langle r\rangle+\langle s \rangle =R\}$ where for all $r\in R$, $\langle r \rangle$ is the ideal generated by $r$. Let $\beta_R:=\{\sigma_r: r\in R\}$, that is, $\sigma_r$ is the set of all elements that are \textit{comaximal} to $r$.  The set $\beta_R$ is a basis for some topology on $R$ \cite[Theorem 2.1]{MACIAS2024109070}. We denote by $\widetilde{M(R)}$ the topological space $(R,\tau_R)$, where $\tau_R$ is the topology generated by $\beta_R$. \textit{The Macías topological space} $M(R)$ is the topological space $(R^0, \tau_{R^0})$ where $\tau_{R^0}$ is the topology generated by the basis $\beta_{R^0}:=\{\sigma_k^0: k\in R^0\}$.

In \cite{clark2019note}, the topological space $\widetilde{G(R)}$ is defined, which is obtained by equipping $R$ with the topology generated by the collection of \textbf{coprime cosets} ${x+I}$ where $x\in R$ and $I$ is a nonzero ideal of $R$. \textit{The Golomb topological space} $G(R)$ is obtained by equipping $R^0$ with the subspace topology induced by $\widetilde{G(R)}$.

The following theorem summarizes the properties of $\widetilde{M(R)}$ and $M(R)$ presented in \cite{MACIAS2024109070}.

\begin{theorem}\label{thpropiedades}
The following propositions hold over $R$:
\begin{enumerate}
\item $\sigma_{rs}=\sigma_r\cap \sigma_s$ for all $r,s\in R$.
    \item $u\in R^\times$ if and only if $\sigma_u=R$.
    \item If $u\in R^\times$, then $u\in\sigma_r$ for all $r\in R$.
    \item $0\in\sigma_k$ if and only if $k\in R^\times$.
    \item $\textbf{cl}_{\widetilde{M(R)}}(\{0\})\subset \textbf{cl}_{\widetilde{M(R)}}(\{r\})$ for all $r\in R$.

    \item $\textbf{cl}_{\widetilde{M(R)}}(\{ru\})=\textbf{cl}_{\widetilde{M(R)}}(\{ur\})=\textbf{cl}_{\widetilde{M(R)}}(\{r\})$ for all $r\in R$ and $u\in R^\times$.
    
    \item If $u\in R^\times$, then $\textbf{cl}_{M(R)}(\{u\})=R^0$ and $\textbf{cl}_{\widetilde{M(R)}}(\{u\})=R$.
    \item $\widetilde{M(R)}$ and $M(R)$ are topological semigroups.
    \item $\widetilde{M(R)}$ and $M(R)$ are hyperconnected.
    \item  Every open set in $\widetilde{M(R)}$ (except for $\sigma_0=R^\times$) is open in $\widetilde{G(R)}$ (The Macías topology is coarser than the Golomb topology).
    \item  $M(R)$ is indiscrete if and only if $R$ is a field.
    \item If R is a principal ideal domain then for any irreducible element $p$ in $R$, $\textbf{cl}_{\widetilde{M(R)}}(\{p\})= \langle p\rangle$. Moreover, for any $x,y\in R$,  $\textbf{cl}_{\widetilde{M(R)}}(\{xy\}) = \textbf{cl}_{\widetilde{M(R)}}(\{x\}) \cap \textbf{cl}_{\widetilde{M(R)}}(\{y\}).$ 

    \item If $R$ is a principal ideal domain, then $\widetilde{M(R)}$ and $M(R)$ are ultraconnected and do not satisfy the separation axiom $\mathrm{T}_0$.
\end{enumerate}
\end{theorem}

\section{Other Closure Properties on $\widetilde{M(R)}$ and $M(R)$}\label{section3}

The results presented in this section were obtained while attempting to answer the following question.
\begin{question}\label{question1}
Let $x \in R$. When does it hold that $\mathbf{cl}_{\widetilde{M(R)}}(\{x\}) = \langle x \rangle$?
\end{question}

The following theorem shows that it is necessary and sufficient for $\langle x \rangle$ to be closed in $\widetilde{M(R)}$.

\begin{theorem}\label{th1}
Let $x \in R$. Then $\mathbf{cl}_{\widetilde{M(R)}}(\{x\}) = \langle x \rangle$ if and only if $\langle x \rangle$ is closed in $\widetilde{M(R)}$. Similarly, if $x \in R^0$, then $\mathbf{cl}_{M(R)}(\{x\}) = \langle x \rangle^0$ if and only if $\langle x \rangle^0$ is closed in $M(R)$.
\end{theorem}
\begin{proof}
    It is clear that $\mathbf{cl}_{\widetilde{M(R)}}(\{x\}) = \langle x \rangle$ implies that $\langle x \rangle$ is closed in $\widetilde{M(R)}$. On the other hand, suppose $\langle x \rangle$ is closed in $\widetilde{M(R)}$. Let $x_0 \in \langle x \rangle$ and $\sigma_r \in \mathcal{B}_R$ such that $x_0 \in \sigma_r$. Then $R = \langle x_0 \rangle + \langle r \rangle \subset \langle x \rangle + \langle r \rangle \subset R$, that is $\langle x \rangle + \langle r \rangle =R$. Thus $x \in \sigma_r$, so $x_0 \in \mathbf{cl}_{\widetilde{M(R)}}(\{x\})$ and $\langle x \rangle \subset \mathbf{cl}_{\widetilde{M(R)}}(\{x\})$. Finally, since $\langle x \rangle$ is closed by hypothesis, $\mathbf{cl}_{\widetilde{M(R)}}(\{x\}) \subset \langle x \rangle$.
\end{proof}

%%%%%%%%%%%%%%%%%%%%%%%%%%%%%%%%%

\begin{lemma}\label{lem}
Let $x \in R$. Then $x \in \sigma_x$ if and only if $x \in R^\times$.
\end{lemma}
\begin{proof}
If $x \in \sigma_x$, then $R = \langle x \rangle + \langle x \rangle = \langle x \rangle$. Therefore, $x \in R^\times$. On the other hand, if $x \in R^\times$, then $x \in \sigma_x$ (see Theorem \ref{thpropiedades} item (3)).
\end{proof}

We know that when $x \in R^\times$, we have $\mathbf{cl}_{\widetilde{M(R)}}(\{x\}) = R$; moreover, $\mathbf{cl}_{\widetilde{M(R)}}(\{0\}) \subset \mathbf{cl}_{\widetilde{M(R)}}(\{x\})$ for all $x \in R$ (see Theorem \ref{thpropiedades}, item (5) and item (7)). The following theorem shows, in general, what happens when $x \in R^\#$.

\begin{theorem}\label{th2}
Let $x \in R^\#$. Then $\langle x \rangle \subset \mathbf{cl}_{\widetilde{M(R)}}(\{x\}) \subset R \setminus \sigma_x$. Similarly, $\langle x \rangle^0 \subset \mathbf{cl}_{M(R)}(\{x\}) \subset R^0 \setminus \sigma^0_x$.
\end{theorem}
\begin{proof}
That $\langle x \rangle \subset \mathbf{cl}_{\widetilde{M(R)}}(\{x\})$ is given by Theorem \ref{th1}. On the other hand, since $x$ is not a unit, $x \notin \sigma_x$ (by Lemma \ref{lem}). Therefore, $R \setminus \sigma_x$ is a closed set containing $x$, so $\mathbf{cl}_{\widetilde{M(R)}}(\{x\}) \subset R \setminus \sigma_x$.
\end{proof}
%%%%%%%%%%%%%%%%%%%%%%%%%%%%%%%%%%%
 One of the authors raised \cite[Question 3.13]{MACIAS2024109070}, does there exist an integral domain $R$ which is not a PID, but the spaces $\widetilde{M(R)}$ and $M(R)$ are ultraconnected? In fact, let $F$ and $G$ be two non-trivial closed sets in $\widetilde{M(R)}$. Let $x \in F$ and $y \in G$. If we have that $\mathbf{cl}_{\widetilde{M(R)}}(\{xy\}) = \mathbf{cl}_{\widetilde{M(R)}}(\{x\}) \cap \mathbf{cl}_{\widetilde{M(R)}}(\{y\})$, then
\begin{equation*}
    \{xy\} \subset \mathbf{cl}_{\widetilde{M(R)}}\{xy\} = \mathbf{cl}_{\widetilde{M(R)}}\{x\} \cap \mathbf{cl}_{\widetilde{M(R)}}\{y\} \subset F \cap G.
\end{equation*}
This condition of the intersection of closure of singletons is sufficient for ultraconnectedness of both spaces. Must note, the previous results hold in general for commutative rings with identity. But for the following results, one required to avoid zero divisors. 

\begin{example}\label{example1}
For $M(R)$, consider $R=\mathbb{Z}_6$, which $M(R)$ is not ultraconnected (see \cite[Example 3.15]{MACIAS2024109070}. For $\widetilde{M(\mathbb{R})}$, consider the commutative ring with identity $R=\mathbb{Z}\times\mathbb{Z}$, which has zero divisors. Consider $\widetilde{M(R)}\simeq \widetilde{M(\mathbb{Z})}\times \widetilde{M(\mathbb{Z})}$ (with the product topology). Note that 
\begin{equation*}
    \mathbf{cl}_{\widetilde{M(R)}}(\{(2,0)\})=2\mathbb{Z}\times \{0\} \ \ \text{and} \ \ \mathbf{cl}_{\widetilde{M(R)}}(\{(0,3)\})=\{0\} \times 3\mathbb{Z},
\end{equation*}
which implies that 
\begin{equation*}
    \mathbf{cl}_{\widetilde{M(R)}}(\{(2,0)\})\cap \mathbf{cl}_{\widetilde{M(R)}}(\{(0,3)\}) =\emptyset .
\end{equation*}
On the other hand,
\begin{equation*}
    \mathbf{cl}_{\widetilde{M(R)}}(\{(2,0)\cdot (0,3)\})=\mathbf{cl}_{\widetilde{M(R)}}(\{(0,0)\})=\{0\}\times \{0\},
\end{equation*}
so that
\begin{equation*}
    \mathbf{cl}_{\widetilde{M(R)}}(\{(2,0)\})\cap \mathbf{cl}_{\widetilde{M(R)}}(\{(0,3)\})\neq  \mathbf{cl}_{\widetilde{M(R)}}(\{(2,0)\cdot (0,3)\}).
\end{equation*}
\end{example}

%%%%%%%%%%%%%%%%%%%%%%%%%%%%%%%%%%%%%%%
\begin{theorem}\label{cor}
Let $x, y \in R$. Then $\mathbf{cl}_{\widetilde{M(R)}}(\{xy\}) = \mathbf{cl}_{\widetilde{M(R)}}(\{x\}) \cap \mathbf{cl}_{\widetilde{M(R)}}(\{y\})$. Similarly, for $x, y \in R^0$, we have that $\mathbf{cl}_{M(R)}(\{xy\}) = \mathbf{cl}_{M(R)}(\{x\}) \cap \mathbf{cl}_{M(R)}(\{y\})$.
\end{theorem}
%%%%%%%%%%%%%%%%%%%%%%%%%%%%%%%%%%%

%%%%%%%%%%%%%%%%%%%%%%%%%%%%%%%%%%%
\begin{proof}
By Theorem \ref{th2}, $\langle x \rangle \subset \mathbf{cl}_{\widetilde{M(R)}}(\{x\})$ and $\langle y \rangle \subset \mathbf{cl}_{\widetilde{M(R)}}(\{y\})$. Note that $xy \in \langle xy \rangle \subset \langle x \rangle \subset \mathbf{cl}_{\widetilde{M(R)}}(\{x\})$. Similarly, $xy \in \mathbf{cl}_{\widetilde{M(R)}}(\{y\})$. Thus, $xy \in \mathbf{cl}_{\widetilde{M(R)}}(\{x\}) \cap \mathbf{cl}_{\widetilde{M(R)}}(\{y\})$, and therefore, $\mathbf{cl}_{\widetilde{M(R)}}(xy) \subset \mathbf{cl}_{\widetilde{M(R)}}(\{x\}) \cap \mathbf{cl}_{\widetilde{M(R)}}(\{y\})$. On the other hand, let $r \in \mathbf{cl}_{\widetilde{M(R)}}(\{x\}) \cap \mathbf{cl}_{\widetilde{M(R)}}(\{y\})$. Let $\sigma_k \in \mathcal{B}_R$ such that $r \in \sigma_k$. Then $x, y \in \sigma_k$, and consequently $xy \in \sigma_k$. Thus, $r \in \mathbf{cl}_{\widetilde{M(R)}}(\{xy\})$, and therefore $\mathbf{cl}_{\widetilde{M(R)}}\{x\} \cap \mathbf{cl}_{\widetilde{M(R)}}\{y\} \subset \mathbf{cl}_{\widetilde{M(R)}}\{xy\}$.
\end{proof}

Theorem \ref{cor} answered \cite[Question 3.13]{MACIAS2024109070}, and Example \ref{example1} showed that it cannot be improved.

\begin{corollary}
The spaces $\widetilde{M(R)}$ and $M(R)$ are ultraconnected. Therefore are normal, limit point compact, pseudocompact and path-connected.
\end{corollary}

% Given a monoid (or semigroup) $S$, we say that a subset $A$ of $S$ is a \textit{monoid ideal} if $SA \subset A$ and $AS \subset A$.

% \begin{theorem}
% Let $x \in R$. Then $\mathbf{cl}_{\widetilde{M(R)}}(\{x\})$ is a monoid ideal in the semigroup $(R, \cdot)$.
% \end{theorem}

% \begin{proof}
% Let $r \in \mathbf{cl}_{\widetilde{M(R)}}(\{x\})$. Let $\alpha \in R$. Let $\sigma_k \in \mathcal{B}_R$ such that $\alpha r \in \sigma_k$. Clearly, $r \in \sigma_k$ ($R = \langle \alpha r \rangle + \langle k \rangle \subset \langle r \rangle + \langle k \rangle \subset R$). Thus, by hypothesis, $x \in \sigma_k$. Therefore, $\alpha r = r \alpha \in \mathbf{cl}_{\widetilde{M(R)}}(\{x\})$.
% \end{proof}

Let $\mathfrak{M}$ be the set of maximal ideals of $R$. For $x \in R^\times$, we denote by $\mathbb{j}(\langle x \rangle)$ the intersection of all maximal ideals that contain $\langle x \rangle$, that is,

\begin{equation*}
    \mathbb{j}(\langle x\rangle)=\bigcap_{\substack{\mathfrak{m}\in\mathfrak{M}\\ \langle x\rangle\subset\mathfrak{m}}}\mathfrak{m}.
\end{equation*}
Most note that $\mathbb{j}(\langle x\rangle)$ was called by Gilmer \cite{gilmer1992multiplicative} the Jacobson Radical of $\langle x\rangle$. On the other hand, Kaplansky \cite{kaplansky2006commutative} and Hungerford \cite{hungerford2012algebra} defined Jacobson Radical of $\langle x\rangle$ to be the intersection of the Jacobson Radical of $R$ with $\langle x\rangle$. Both notions are not the same, as in $\mathbb{Z}$ for any two primes $p_1, p_2$, $\mathbb{j}(\langle p_1 p_2\rangle)$ it is not zero. On the other hand, $J(\mathbb{Z})\cap \langle x\rangle =\{0\}\cap \langle x\rangle=\{0\}$.

\begin{proposition}\label{propclausure}
Let $x \in R^\times$. Then $\mathbf{cl}_{\widetilde{M(R)}}(\{x\}) = \mathbb{j}(\langle x \rangle)$. Similarly, let $x\in R^\#$. Then $\mathbf{cl}_{M(R)}(\{x\}) = \mathbb{j}(\langle x \rangle)^0$.
    
\end{proposition}

\begin{proof}
Let $r \in \mathbb{j}(\langle x\rangle)$. Let $\sigma_k \in \mathcal{B}_R$ such that $r \in \sigma_k$. \textbf{Claim: } $\langle x \rangle + \langle k \rangle = R$. If this were not the case, then there would exist a maximal ideal $\mathfrak{m}$ such that $\langle x \rangle + \langle k \rangle \subset \mathfrak{m}$, but then $\langle x \rangle \subset \mathfrak{m}$, so $r \in \mathfrak{m}$. Thus, $R = \langle r \rangle + \langle k \rangle \subset \mathfrak{m}$, which would be absurd. Therefore, the \textbf{Claim} is true, which implies that $x \in \sigma_k$ and thus $r \in \mathbf{cl}_{\widetilde{M(R)}}(\{x\})$. Hence, $\mathbb{j}(\langle x\rangle) \subset \mathbf{cl}_{\widetilde{M(R)}}(\{x\})$. On the other hand, let $r \in R \setminus \mathbb{j}(\langle x\rangle)$ ($r \notin \mathbb{j}(\langle x\rangle)$). Then, there exists a maximal ideal $\mathfrak{m}$ such that $\langle x \rangle \subset \mathfrak{m}$ and $r \notin \mathfrak{m}$. Since $\mathfrak{m}$ is maximal, the coset $r + \mathfrak{m}$ is invertible, meaning that $\langle r \rangle + \mathfrak{m} = R$, so there exists $\alpha \in R$ and $m \in \mathfrak{m}$ such that $\alpha r + m = 1$. Thus, $r \in \sigma_m$. Now, note that $\sigma_m \cap \mathfrak{m} = \emptyset$, because if there existed $y \in \sigma_m \cap \mathfrak{m}$, then $R = \langle y \rangle + \langle m \rangle \subset \mathfrak{m}$, which is absurd. Thus, $\sigma_m \cap \langle x \rangle = \emptyset$, so $x \notin \sigma_m$. Therefore, $r \notin \mathbf{cl}_{\widetilde{M(R)}}(\{x\})$. Hence, $\mathbf{cl}_{\widetilde{M(R)}}(\{x\}) \subset \mathbb{j}(\langle x\rangle)$.
\end{proof}

Notices that $\mathbf{cl}_{\widetilde{M(R)}}(\{x\})$ is an ideal as the intersection of ideal is an ideal. And by the Chinese Remainder Theorem implies that $R/\mathbf{cl}_{\widetilde{M(R)}}(\{x\})\cong \prod_{\substack{\mathfrak{m}\in\mathfrak{M}\\ \langle x\rangle\subset\mathfrak{m}}} R/\mathfrak{m}$. Hence, if an element does not belong to a maximal ideal (containing $\langle x\rangle$), such element isn't in the clousure of $\{x\}$ in $\widetilde{M(R)}$. Equivalently, if it does not belong to the kernel of the canonical projection of $R$ onto a field of the form $R/\mathfrak{m}$ with $\mathfrak{m}$ a maximal ideal of $R$, it does not belong to the clousure of $\{x\}$.
%%%%%%%%%%%%%%%%%%%%%%%%%%%%%%%%%%%%%%%%%%
%\begin{corollary}
%For every $x\in R$, $\mathbf{cl}_{\widetilde{M(R)}}(\{x\})$ is an ideal of $R$ .
%\end{corollary}
%%%%%%%%%%%%%%%%%%%%%%%%%%%%%%%%%%%%%

The \textit{radical} of of an ideal $I$ in $R$,  denoted by $\sqrt{I}$, is defined as

\begin{equation*}
    \sqrt{I}=\{r\in R: r^n\in I \ \ \text{for some} \ \ n\in \mathbb{N}\}.
\end{equation*}

\begin{corollary}
Let $x\in R\setminus R^\times$.  Then $\sqrt{\langle x\rangle}\subset \mathbf{cl}_{\widetilde{M(R)}}(\{x\})$. Similarly, let $x\in R^\#$. Then  $\sqrt{\langle x\rangle}^0\subset \mathbf{cl}_{M(R)}(\{x\})$.
\end{corollary}

\begin{corollary}
If $R$ is a principal ideal domain, then $\mathbf{cl}_{\widetilde{M(R)}}(\{p\})=\sqrt{\langle p\rangle}$ for every irreducible element $p\in R$. Similarly, $\mathbf{cl}_{M(R)}(\{p\})=\sqrt{\langle p\rangle}^0$.
\end{corollary}
 %%%%%%
\begin{corollary}
If $x$ is prime in $R$, then $\mathbf{cl}_{\widetilde{M(R)}}(\{x\})={\langle x\rangle}$.
\end{corollary}
%%%%%%%%%%%%%%%%%%%%%%%%%%%%%%%%%%%%%%%%%%%
\begin{proposition}
Let $I$ be a proper ideal of $R$. Then $cl_{\widetilde{M(R)}}(I) = \mathbb{j}(I)$. Similarly, $cl_{M(R)}(I^0) = \mathbb{j}(I)^0$.
\end{proposition}

\begin{proof}
    Replace $\langle x \rangle$ with $I$ in the proof of Proposition \ref{propclausure}.
\end{proof}

%%%%%%%%%%%%%%%%%%%%%%%%%%%%%%%%%%%%%
%%%%%%%%%%%%%%%%%%%%%%%%%%%%%%%%%%%%%%%
We say two ideals $I$ and $J$ of $R$ are said to be \textit{comaximal}, if the ideal generated by $I$ and $J$ is not proper, that is $I+J=R$. On the other hand, we say that $x,y\in R$ are \textit{coprime}, if they do not have any common nonzero nounit factor or the greatest common factor is $1$ (or a unit). The comaximality always implies coprimality, but it is well known that the converse is not always true. For example, the greatest common factor of $x$ and $2$ is $1$ in $\mathbb{Z}[x]$, but $1$ cannot be represented as a linear combination of $x$ and $2$. In the terms of algebraic structures, the above example indicates that a Bezout domain (an integral domain in which every finitely generated ideal is principal), is a GCD domain (an integral domain with the property that for any two elements $x$ and $y$, not both the additive identity, their greatest common factor exist).
\begin{theorem}\label{th3}
Let $R$ be an integral domain where coprimality implies comaximality ($\gcd(x,y) = 1 \Rightarrow \langle x \rangle + \langle y \rangle = R$). If $p \in R$ is a prime element, then $R \setminus \sigma_p = \langle p \rangle$.
\end{theorem}
\begin{proof}
By Theorem \ref{th2}, we know that $\langle p \rangle \subset R \setminus \sigma_p$. Now, if $x_0 \in R \setminus \sigma_p$, then $x_0 \notin \sigma_p$. That is, $x_0$ is not comaximal with $p$, and hence $x_0$ is not coprime with $p$. Since $p$ is prime, it necessarily follows that $x_0 \in \langle p \rangle$. Therefore, $R \setminus \sigma_p \subset \langle p \rangle$.
\end{proof}

\begin{example}
Theorem \ref{th3}, along with Theorem \ref{th2}, gives us a \textit{more direct} proof of \cite[Theorem 3.1]{MACIAS2024109070}. Of course, if $R$ is a principal ideal domain, then coprimality implies comaximality. Therefore, in this case, for every irreducible element $p \in R$, by virtue of the aforementioned theorems, $\mathbf{cl}_{\widetilde{M(R)}}(\{p\}) = \langle p \rangle$.
\end{example}

\begin{theorem}\label{th4}
Let $R$ be an integral domain with the property that for every element $x \in R^\#$, it holds that $R \setminus \sigma_x = \langle x \rangle$, then coprimality implies comaximality.
\end{theorem}

\begin{proof}
Let $x,y \in R^\#$ be such that $\gcd(x,y) = 1$. If $x$ and $y$ are not comaximal, then $x \notin \sigma_y$ and $y \notin \sigma_x$. Then, $x \in R \setminus \sigma_y = \langle y \rangle$ and $y \in R \setminus \sigma_x = \langle x \rangle$, which cannot be possible, because $x$ and $y$ are coprime. Therefore, coprimality implies comaximality.
\end{proof}
%%%%%%%%%%%%%%%%%%%%%%%%%%%%%%%%%%%%%%
Let us recall that an atomic domain is an integral domain with the property that every nonzero nonunit element can be written as a finite product of irreducible elements.
\begin{theorem}\label{th5}
Let $R$ be an atomic domain. If for every irreducible element $p \in R$, it holds that $R \setminus \sigma_p = \langle p \rangle$, then coprimality implies comaximality.
\end{theorem}

\begin{proof}
Suppose that $R \setminus \sigma_p = \langle p \rangle$ for every irreducible element $p \in R$. Let $x, y \in R^\times$ be coprime ($\gcd(x,y) = 1$). Since $R$ is a unique factorization domain, we can write $x = p_1^{\alpha_1} \cdot p_2^{\alpha_2} \cdots p_n^{\alpha_n}$ and $y = q_1^{\beta_1} \cdot q_2^{\beta_2} \cdots q_m^{\beta_m}$. Since $x$ and $y$ are coprime, each $p_i$ is coprime to each $q_j$. Now, by Theorem \ref{thpropiedades}, item (1), we have that $R \setminus \sigma_x = \bigcup_{i=1}^n R \setminus \sigma_{p_i}$, and similarly, $R \setminus \sigma_y = \bigcup_{i=1}^m R \setminus \sigma_{q_i}$. Then, by hypothesis, $R \setminus \sigma_x = \bigcup_{i=1}^n \langle p_i \rangle$ and $R \setminus \sigma_y = \bigcup_{i=1}^m \langle q_i \rangle$. If we assume that $x$ and $y$ are not comaximal, then $x \notin \sigma_y$ and $y \notin \sigma_x$, which would imply that $x \in \langle q_i \rangle$ for some $i$ and $y \in \langle p_j \rangle$ for some $j$, which is absurd. Therefore, coprimality implies comaximality. The converse follows from Theorem \ref{th3}.
\end{proof}

Notice that in the previous proof, there is no need to assume that the factorizations in irreducible elements are finite. Hence, we can use the atomic domain as defined by PM Cohn in \cite{cohn1968bezout}. Moreover, if $R$ is a unique factorization domain, the converse also holds. On the other hand, it is well known that if $R$ is a unique factorization domain, then $R$ is a principal ideal domain if and only if coprimality implies comaximality.

\begin{corollary}\label{corcara}
If $R$ is a unique factorization domain, then $R$ is a principal ideal domain if and only if for every irreducible element $p \in R$, it holds that $\langle p \rangle = \mathbf{cl}_{\widetilde{M(R)}}(\{p\}) = R \setminus \sigma_p$.
\end{corollary}

To conclude this section, we would like to propose the following problems.

\begin{problem}\label{pr1}
We know that in a principal ideal domain, the closure in $\widetilde{M(R)}$ of an irreducible is the ideal generated by the irreducible. In what other structures does this hold?
\end{problem}

\begin{problem}\label{pr2}
Is there an $R$ such that every principal ideal is closed in $\widetilde{M(R)}$? Note, for example, that if $R = \mathbb{Z}$, then $\langle 8 \rangle$ is not closed in $\widetilde{M(R)}$, since $\mathbf{cl}_{\widetilde{M(R)}}(\{8\}) =\mathbf{cl}_{\widetilde{M(R)}}(\langle 8\rangle)= \langle 2 \rangle $.
\end{problem}
%%%%%%%%%%%%%%%%%%%%%%%%%%%%%%

\section{On the Infinitude of Non-associated Irreducible Elements and Maximal Ideals}\label{section4}

In this section, we aim to present a topological proof (using the space $M(R)$) of the following propositions.
\begin{proposition}\cite[Theorem 14]{knopemacher1997topologies}\label{pro1}
If $R$ is such that $\# R^\times < \# R$, then the set of maximal ideals of $R$ is infinite.
\end{proposition}

\begin{proposition}\cite[Theorem 3.1]{clark2017euclidean}\label{pro2}
If $R$ is a \textit{Furstenberg domain} with $\# R^\times < \# R$, then $R$ has an infinite number of non-associated irreducible elements. 
\end{proposition}

It is worth mentioning that we will use a completely different argument than the one presented in \cite[Subsection 3.16]{MACIAS2024109070}. Since it does not used the results in Corollary \ref{corcara}, that forces $R$ to be a principal ideal domain.  Furthermore, as suggested by those same results, there is not much control over the closure operator in arbitrary integral domains (that are not fields).

Before we proceed with the proofs of the propositions, consider the following results.

\begin{theorem}\label{th6}
If the set $R^\times$ is not open in $\widetilde{M(R)}$, then the set $R^\#$ is dense in $\widetilde{M(R)}$. Similarly, $R^\#$ is not dense in $M(R)$.
\end{theorem}
\begin{proof}
    If there exists $\sigma_r \in \mathcal{B}_R$ such that $\sigma_r \cap R^\# = \emptyset$, then $\sigma_r \subset R^\times$. Thus $R^\times = \sigma_r$ (see Theorem \ref{thpropiedades} item (3)) would be open in $\widetilde{M(R)}$, which is a contradiction. Thus, $R^\#$ is dense in $\widetilde{M(R)}$.
\end{proof}

\begin{theorem}\label{th7}
If $\# R^\times < \# R$, then the set $R^\times$ is not open in $\widetilde{M(R)}$. Similarly, $R^\times$ is not open in $M(R)$.
\end{theorem}
\begin{proof}
 Let $k \in R^0$. Since $R$ has no zero divisors, the function $\iota: R \to \sigma_k$ given by $r \to 1 + rk$ is an injection, and so $|R| \leq |\sigma_k| \leq |R|$. Consequently, every open set in $\widetilde{M(R)}$ has the cardinality of $R$, so $R^\times$ cannot be open in $\widetilde{M(R)}$.
\end{proof}

These results lead to the following lemmas, which plays an important role in the proof of the propositions.

\begin{lemma}\label{lem 1}
If $\# R^\times < \# R$, then the set $R^\#$ is dense in $\widetilde{M(R)}$ and in $M(R)$.
\end{lemma}
%%%%%%%%%%%%%%%%%%%%%%%%%%%%%%%%%%%%%%%%%%%%%%%%

\begin{lemma}\cite[Problem 25A]{willard2012general}\label{lm4}
The union of finitely many nowhere dense subsets of a topological space is nowhere dense.
\end{lemma}

\begin{proof}

Let $Y$ be a topological space. Let $N_1, N_2, \dots, N_k$ be nowhere dense subsets in $Y$. Consider the set $N := \displaystyle\bigcup_{i=1}^k N_i$. Let $\mathcal{O}$ be a non-empty open set in $Y$. Then, by the definition of a nowhere dense set, for $N_1$ there exists a non-empty open set $\mathcal{O}_1$ in $Y$ such that $\mathcal{O}_1 \subset \mathcal{O}$ and $N_1 \cap \mathcal{O}_1 = \emptyset$. Similarly, for $\mathcal{O}_1$ there exists a non-empty open set $\mathcal{O}_2 \subset \mathcal{O}_1$ such that $\mathcal{O}_2 \cap N_2 = \emptyset$. Continue this process to construct the chain of non-empty open sets
\begin{equation*}
    \mathcal{O}_k \subset \mathcal{O}_{k-1} \subset \cdots \subset \mathcal{O}_2 \subset \mathcal{O}_1 \subset \mathcal{O},
\end{equation*}
such that $\mathcal{O}_i \cap N_i = \emptyset$ for each $i = 1, 2, \dots, k$. Note that, in this way, $\mathcal{O}_k \cap N_i = \emptyset$ for each $i = 1, 2, \dots, k$ and $\mathcal{O}_k \subset \mathcal{O}$. Consequently, $N \cap \mathcal{O}_k = \emptyset$. Therefore, $N$ is nowhere dense in $Y$.
\end{proof}

%%%%%%%%%%%%%%%%%%%%%%%%%%%%%%%%%%%%%%%
\subsection{Proof of Proposition \ref{pro1}}\label{subsection4.1}

We begin with the following lemma.

\begin{lemma}\label{lem2}
Let $I$ be a proper ideal of $R$. Then $I$ is nowhere dense in $\widetilde{M(R)}$. Similarly, $I^0$ is nowhere dense in $M(R)$.
\end{lemma}
\begin{proof}
Suppose that $I$ is dense in $\widetilde{M(R)}$. Let $x \in I$. Since $I$ is dense in $\widetilde{M(R)}$, $\sigma_x \cap I \neq \emptyset$. Let $y \in \sigma_x \cap I$. Then,
\begin{equation*}
R = \langle y \rangle + \langle x \rangle \subset \langle x \rangle + \langle x \rangle = \langle x \rangle \subset I \subset R,
\end{equation*}
which implies that $I = R$, a contradiction. Therefore, $I$ cannot be dense in $\widetilde{M(R)}$. Since $\widetilde{M(R)}$ is hyperconnected (see Theorem \ref{thpropiedades} item (9)), $I$ must necessarily be nowhere dense in $\widetilde{M(R)}$.
\end{proof}

Now we are ready for the proof of Proposition \ref{pro1}, for which we will work with $M(R)$.

\begin{proof}
Let $\mathfrak{M}$ be the set of maximal ideals of $R$. Every element of $R$ that is not a unit belongs to some maximal ideal of $R$. Thus, we can write

\begin{equation*}
    R^0 = R^\times \cup \bigcup_{\mathfrak{m} \in \mathfrak{M}} \mathfrak{m}^0.
\end{equation*}
Then,

\begin{equation*}
    R^\# = \bigcup_{\mathfrak{m} \in \mathfrak{M}} \mathfrak{m}^0.
\end{equation*}

If $\mathfrak{M}$ were finite, by Lemma \ref{lm4} and Lemma \ref{lem2}, $R^\#$ would be nowhere dense, which is absurd by Lemma \ref{lem 1}.
\end{proof}

Note that by using Theorem \ref{th6} and Lemma \ref{lem2}, we obtain the following theorem.

\begin{theorem}
In general, if $R^\times$ is not open in $M(R)$, then the set of maximal ideals of $R$ is infinite.
\end{theorem}
%%%%%%%%%%%%%%%%%%%%%%%%%%%%%%%%%%%%%%%

\subsection{Proof of Proposition \ref{pro2}}\label{subsection4.2}

We begin with a definition followed by a lemma.

\begin{definition}\cite{clark2017euclidean}
    A Furstenberg domain is a commutative ring with a multiplicative identity distinct from the associative identity, without zero divisors, in wich every nonzero nonunit is divisible by an irreducible element.
\end{definition}

\begin{lemma}\label{lemadenso}
The set $ R^\#$ is equal to the set $\{R^0 \setminus \sigma_p^0: p \in \mathcal{I}\}$ where $\mathcal{I}$ is the set of non-associated irreducible elements of $R$. 
\end{lemma}

\begin{proof}
Let $x \in \langle p \rangle^0$ with $p \in \mathcal{I}$. If $x \notin R^0 \setminus \sigma_p^0$, then $R = \langle x \rangle + \langle p \rangle \subset \langle p \rangle \subset R$. This would imply that $p \in R^\times$, which is a contradiction. Therefore, $ \langle p \rangle \setminus \{0\} \subset R^0 \setminus \sigma_p^0$ for all $p \in \mathcal{I}$. On the other hand, it is clear that $\{R^0 \setminus \sigma_p^0: p \in \mathcal{I}\} \subset  R^\#$ since $R^\times \subset \sigma_r^0$ for all $r \in R^0$ (see Theorem \ref{thpropiedades} item (3)), in particular for all $p \in \mathcal{I}$. Now, since $R$ is a Furstenberg domain, we have $ R^\# = \{\langle p \rangle^0: p \in \mathcal{I}\}$ (see \cite[Lemma 3.3]{clark2017euclidean}) and so $R^\# \subset \{R^0 \setminus \sigma_p^0: p \in \mathcal{I}\}$, 
from which the equality follows.
\end{proof}

Now we are ready for the proof of Proposition \ref{pro2}, for which we will work with $M(R)$.

\begin{proof}
    Note that for each $p \in \mathcal{I}$, $R^0 \setminus \sigma_p^0$ is a non-trivial closed set in $M(R)$ and therefore it is nowhere dense in $M(R)$ since $M(R)$ is hyperconnected (see Theorem \ref{thpropiedades} item (9)). Now, if $R$ had a finite number of non-associated irreducible elements, then $R^\#$ would be nowhere dense in $M(R)$, as it would be the finite union of nowhere dense sets in $M(R)$ (see Lemma \ref{lemadenso} and Lemma \ref{lm4}), which is not possible by Lemma \ref{lem2}. Hence, $R$ has an infinite number of non-associated irreducible elements. \end{proof}

Note that by using Theorem \ref{th6} and Lemma \ref{lem2}, the following theorem is obtained.

\begin{theorem}
In general, if $R^\times$ is not open in $M(R)$ and $R$ is a Furstenberg domain, then $R$ has an infinite number of non-associated irreducible elements. Hence, $R$ can not be a Cohen-Kaplansky domain.
\end{theorem}
We can not drop the hypothesis of the cardinality of units need to be strickly less than the cardinality of the integral domain. For example, consider the integer localized at $\langle 2 \rangle$, it an atomic domain (hence a Furstenberg domain) and it has only one irreducible element represented by $2$.

% \subsection{Final comment}
% In \cite{knopemacher1997topologies}, a proof of Proposition \ref{pro1} using $G(R)$ can be found. The proof we present with $M(R)$ cannot be replicated in $G(R)$ because maximal ideals in $R^0$ are not necessarily nowhere dense in $G(R)$. For example, consider $R = \mathbb{Z}$. Suppose all maximal ideals of $R^0$ are nowhere dense in $G(R)$. Como $3$ es primo, entonces $\langle 3\rangle^0$ es un ideal maximal de $R^0$. Luego, por suposición, $\langle 3 \rangle^0$ sería denso en ninguna parte en $G(R)$. Además, $\langle 3 \rangle^0$ es cerrado en $G(R)$ pues $R^0\setminus\langle 3 \rangle^0=(1+\langle 3 \rangle)\cup (2+\langle3\rangle)$.  For further details about the properties of the Golomb topology used in this argument, see \cite[Counterexample 61]{steen1978counterexamples}. Using the same argument, and given that the Furstenberg topology on integral domains (see \cite{knopemacher1997topologies} and \cite{clark2017euclidean} for more details on this topology) is coarser than the Golomb topology, the proof presented for Proposition \ref{pro2} with $M(R)$ cannot be replicated using the Furstenberg topology on Furstenberg domains. In conclusion, the proofs presented in this section are not only different, but are also specific to $M(R)$.

\section{On Ring Isomorphisms and Homeomorphisms}

In this section, $S$ will denote an integral domain that is not a field.  Our goal is to motivate the \textit{problem of homeomorphisms} on $\widetilde{M(R)}$, and in particular on $M(R)$. We begin with the following lemma.

\begin{lemma}\label{lemaiso1}
Let $\phi: R \to S$ be a ring isomorphism. Let $\sigma_k \in \mathcal{B}_R$. Then $r \in \sigma_k$ if and only if $\phi(r) \in \sigma_{\phi(k)}$ (here $\sigma_{\phi(k)} \in \mathcal{B}_S$). Consequently, $\phi(\sigma_k) = \sigma_{\phi(k)}$.
\end{lemma}

\begin{proof}
If $r \in \sigma_k$, then $\langle r \rangle + \langle k \rangle = R$. Since $\phi$ is a ring homeomorphism, then
\begin{equation*}
    S = \phi(R) = \phi(\langle r \rangle + \langle k \rangle) = \langle \phi(r) \rangle + \langle \phi(k) \rangle.
\end{equation*}

Thus, $\phi(r) \in \sigma_{\phi(k)}$. Conversely, if $\phi(r) \in \sigma_{\phi(k)}$, then $\langle \phi(r) \rangle + \langle \phi(k) \rangle = S$. Let $\phi^{-1}: S \to R$ be the inverse of $\phi$ (which is also a ring isomorphism). Then
\begin{equation*}
    R = \phi^{-1}(S) = \phi^{-1}(\langle \phi(r) \rangle + \langle \phi(k) \rangle) = \langle \phi^{-1}(\phi(r)) \rangle + \langle \phi^{-1}(\phi(k)) \rangle = \langle r \rangle + \langle k \rangle.
\end{equation*}

Thus, $r \in \sigma_k$.
\end{proof}

\begin{theorem}\label{thisoI}
If $R \cong S$, then $\widetilde{M(R)} \simeq \widetilde{M(S)}$. Similarly, $M(R) \simeq M(S)$.
\end{theorem}

\begin{proof}
Since $R \cong S$, there exists a ring isomorphism $\phi: R \to S$. Let $\sigma_s \in \mathcal{B}_S$. Note that for $s$, there exists $r \in R$ such that $\phi(r) = s$. Then, using Lemma \ref{lemaiso1}, we have
\begin{equation*}
    \phi^{-1}(\sigma_s) = \phi^{-1}(\sigma_{\phi(r)}) = \sigma_{\phi^{-1}(\phi(r))} = \sigma_r.
\end{equation*}

Thus, $\phi$ is continuous, and therefore $\widetilde{M(R)} \simeq \widetilde{M(S)}$.
\end{proof}

\begin{corollary}
     Let $h: R\to S$ be a ring homomorphism and $K:=ker(h)$. Then $\widetilde{M(R/K)}\simeq \widetilde{M(h(R))}$.
\end{corollary}
Theorem \ref{thisoI} essentially states that \textit{isomorphisms induce homeomorphisms}. However, the converse does not generally hold, as illustrated in the following example.

\begin{example}
    Let $R = S = \mathbb{Z}$. Define $h: M(\mathbb{Z}) \to M(\mathbb{Z})$ by
    \begin{equation*}
    h(n) = \left\{ \begin{array}{lcc} 
        4 & \text{if} & n = 2 \\ \\ 
        2 & \text{if} & n = 4 \\ \\ 
        n & \text{otherwise.} 
    \end{array} \right.
    \end{equation*}
    It is easy to see that $h$ is not an isomorphism. However, it is a homeomorphism (see \cite[Example 3.1]{macias2024self}).
\end{example}

Finally, we present the following result.

\begin{theorem}
    Let $h: \widetilde{M(R)} \to \widetilde{M(S)}$ be a homeomorphism. Then $h(R^\times) = S^\times$.
\end{theorem}

\begin{proof}
    Let $u \in R^\times$. We will show that $h(u) \in S^\times$. Assume $h(u) \notin S^\times$. By Proposition \ref{propclausure}, we have
    \begin{equation*}
        \textbf{cl}_{\widetilde{M(S)}}(\{h(u)\}) = \mathbb{j}(\langle h(u) \rangle) \subsetneq S.
    \end{equation*}
    On the other hand, using the fact that $h$ is a homeomorphism and Theorem \ref{thpropiedades}, we get
    \begin{equation*}
        \textbf{cl}_{\widetilde{M(S)}}(\{h(u)\}) = h(\textbf{cl}_{\widetilde{M(R)}}(\{u\})) = h(R) = S.
    \end{equation*}
    Thus, $S \subsetneq S$, which is absurd. Therefore, $h(u) \in S^\times$, and so $h(R^\times) \subset S^\times$.

    Now note that this argument holds for \textit{any homeomorphism}, in particular for $h^{-1}: \widetilde{M(S)} \to \widetilde{M(R)}$. Let $u \in S^\times$. Then $u = h(u^*)$ for some $u^* \in R$. Suppose $u^* \notin R^\times$. Then $h^{-1}(u) = h^{-1}(h(u^*)) = u^* \notin R^\times$, which is a contradiction. Hence, $u \in h(R^\times)$, and thus $S^\times \subset h(R^\times)$.
\end{proof}

Motivated by this last result, it is worth mentioning that the results presented in \cite{macias2024self} can be reproduced (generalized) for infinite principal ideal domains. 

Finally, we post this last question about the converse of \ref{thisoI}.

\begin{problem}
    When does a homeomorphism of $\widetilde{M(S)}$ and $\widetilde{M(R)}$ determine an isomorphism of the rings $S$ and $R$?
\end{problem}

\printbibliography
\end{document}